\documentclass{amsart}
\usepackage{amsmath}
\usepackage{amsthm}
\usepackage{amsfonts}
\usepackage{enumerate}
\usepackage{amssymb}
\usepackage{amsrefs}
\usepackage{amsthm}
\usepackage{hyperref}
\usepackage{array}
\newtheorem*{thm}{Theorem}
\newtheorem*{thm1}{Theorem 1}
\newtheorem*{thm2}{Theorem 2}

\newtheorem*{defin}{Definition}
\newtheorem*{prop}{Proposition}

\newtheorem{cor}{Corollary}
\newtheorem{prob}{Problem}
\newtheorem{example}{Example}
\numberwithin{equation}{section}

\BibSpec{arXiv}{%
  +{}{\PrintAuthors}{author}
  +{,}{ \textit}{title}
  +{}{ \parenthesize}{date}
  +{,}{ arXiv }{eprint}
  +{,}{ primary class }{primaryclass}
}

\begin{document}
\title{Exponential Hilbert series and hierarchical log-linear models}
\author{Wayne A. Johnson}
\address{Department of Statistics\\
 Truman State University\\
 wjohnson@truman.edu}

\begin{abstract}
Consider a hierarchical log-linear model for a discrete random vector, $X$, given by a simplicial complex, $\Gamma$, and an integer matrix, $A_\Gamma$. We give a new characterization of the rank of $A_\Gamma$ given by a logarithmic transformation on the exponential Hilbert series of $\Gamma$. We show that, if each random variable in $X$ has the same number of possible outcomes, then this formula reduces to a simple description in terms of the face vector of $\Gamma$. If $\Gamma$ further satisfies the Dehn-Sommerville relations, then we give an exceptionally simple formula for computing the rank of $A_\Gamma$, and thus the dimension and the number of degrees of freedom of the model.
\vspace{1 pc}
\end{abstract}

\maketitle

\section{Introduction}

In this note, we show how the exponential Hilbert series, as defined in \cite{JM} can be used to compute the dimension (and thus the number of degrees of freedom) of a hierarchical log-linear model. We focus on how, under a simple logarithmic transformation, the rank formula first computed in \cite{HS} can be obtained through the formula for the exponential Hilbert series of the underlying simplicial complex computed in \cite{JM}. Under certain assumptions, it can further be shown that the formula from \cite{HS} can be computed using the face vector of the simplicial complex, which greatly simplifies the computation.

Throughout the rest of this section, we provide background on log-linear models from the point of view of algebraic statistics, as well as a quick primer on simplicial complexes. The former borrows heavily from \cite{Sull}. More detail on simplicial complexes and Hilbert series can be found in \cite{MS}. In \S2, we present more focused background material on hierarchical log-linear models. In \S3, we present material on exponential Hilbert series, following \cite{JM}. The remaining sections are a discussion of the main results with a quick application to cyclic models, a presentation of some interesting examples, and some open questions related to Hilbert series and hierarchical models.

\subsection{Log-linear models}

Let $r\in\mathbb{Z}^+$. We denote by $\Delta_{r-1}$ the probability simplex of dimension $r-1$. A \textit{log-linear model} is a certain family of probability distributions given by a $k\times r$-integer matrix, $A$.

\begin{defin}
Let $A\in\mathbb{Z}^{k\times r}$ be an integer matrix such that $(1,1,\dots,1):=\textbf{1}$ is in the row span of $A$. The log-linear model associated to $A$ is the set of probability distributions
\begin{center}
$\mathcal{M}_A:=\{p\in\text{int}(\Delta_{r-1}) : \log(p)\in\text{rowspan}(A)\}$.
\end{center}
\end{defin}

Note that $p\in\text{int}(\Delta_{r-1})$ implies that no probability in the distribution $p$ is zero. Log-linear models describe exponential families of discrete random variables with a finite number of states. Let $X=(X_1,\dots,X_k)$ be a discrete random vector with state space $[r]:=\{1,2,\dots,r\}$. Let $T:[r]\rightarrow\mathbb{R}^k$ be a statistic, i.e. a measurable map from $[r]$ to $\mathbb{R}^k$. Then $T$ associates a vector in $\mathbb{R}^k$ to each $x\in[r]$. Denote this vector by $T(x)=a_x$. Assume that each component of $a_x$ is an integer, and let $A$ be the $k\times r$-matrix whose rows are the vectors $a_x$.

In general, an exponential family is determined by its parameter space. For a discrete exponential family, let $\eta:=(\eta_1,\dots,\eta_k)^t$ be a vector of parameters. Then define the probability density function $p_\eta$ on $[r]$ by
\begin{center}
$p_\eta(x):=e^{\eta^tT(x)-\phi(\eta)}$,
\end{center}
where $\phi(\eta)$ is
\begin{center}
$\log\displaystyle\left(1+\sum_{x=1}^{r-1}e^{\eta_x}\right)$. 
\end{center}
Then, taking the logarithm, we get
\begin{equation}
\log(p_\eta(x))=\eta^tT(x)-\phi(\eta).
\end{equation}

Let $\theta_i=e^{\eta_i}$. Then we can rewrite $p_\eta(x)$ as
\begin{center}
$p_\theta(x)=\displaystyle\frac{1}{Z(\theta)}\prod_{j}\theta_j^{a_jx}$,
\end{center}
where $Z(\theta)$ is a normalizing constant, and the integers $a_{jx}$ are the terms in $A$. Then (1.1) becomes
\begin{center}
$\log(p_\theta(x))=\log(\theta)^tA-\log(Z(\theta))$.
\end{center}
If we add the final assumption that $\textbf{1}\in\text{rowspace}(A)$, then this is equivalent to saying that $\log(p)$ is in the row space of $A$. Thus, a log-linear model is the probability model on a discrete exponential family.

As a quick example, let $A=\begin{pmatrix}2 & 1 & 0\\ 0 & 1 & 2\end{pmatrix}$. Then there are two variables, $X_1$ and $X_2$. The model corresponds to the interaction factors
\begin{center}
$\theta_1^2$, $\theta_1\theta_2$, and $\theta_2^2$
\end{center}
corresponding to the model where all degree two interaction factors are non-trivial. These factors can be read off the matrix by thinking of the rows as indexed by $X_1$ and $X_2$ and the columns giving the exponents in the parameters.

\subsection{Simplicial complexes}

This paper concerns a special class of log-linear models whose parameters interact in a specific way determined by a simplicial complex. Here we recall the facts about simplicial complexes we will need in the sequel.

Let $\Gamma$ be an abstract simplicial complex on $[m]$. This means that $\Gamma$ is a collection of subsets of $[m]$ such that $\Gamma$ contains all singletons and is closed under taking subsets. We will alternatively denote $\Gamma$ by listing out all sets in $\Gamma$ or by listing just the maximal subsets contained in $\Gamma$. For example,
\begin{center}
$\Gamma=\{\emptyset, \{1\}, \{2\}, \{1,2\}\}$ \\
$\Gamma=[12]$
\end{center}
denote the same simplicial complex on $[2]$. The subsets contained in $\Gamma$ are called the \textit{faces} of $\Gamma$. As above, we include the empty face. Given a face, $F\in\Gamma$, the dimension of $F$ is $\dim(F):=\# F-1$. The dimension of $\Gamma$ is the dimension of its largest face. The empty set is the unique face of dimension $-1$. If $F$ is a face of dimension $i$, we call $F$ an $i$-face of $\Gamma$.

Abstract simplicial complexes are studied via the algebra of their Stanley-Reisner rings, which is reviewed in \S3. The main invariant of a simplicial complex we will need is the face vector, $f(\Gamma)=(f_{-1},f_0,\dots,f_{\dim(\Gamma)})$, where $f_i$ denotes the number of $i$-faces in $\Gamma$. If $\Gamma$ can be assumed from context, we simply write the face vector as $f$. The face vector has been studied extensively. For a modern treatment, relevant to the results in the sequel, see \cite{NS}.

\section{Hierarchical log-linear models models}

Let $X=(X_1,\dots,X_m)$ be a discrete random vector, and assume that each $X_i$ has state space $[r_i]$. Let
\begin{center}
$\mathcal{R}:=\displaystyle\prod_{i=1}^m[r_i]$
\end{center}
denote the joint state space of the random vector $X$.

We use the convention in \cite{Sull} for writing subindices: let $i=(i_1,\dots,i_m)\in\mathcal{R}$ and $F=\{f_1,f_2,\dots\}\subseteq[m]$. Then we define
\begin{center}
$i_F:=(i_{f_1},i_{f_2},\dots)$.
\end{center}
For each $F\subseteq[m]$, the random subvector $X_F:=(X_f)_{f\in F}$ has state space $\mathcal{R}_F=\displaystyle\prod_{f\in F}[r_f]$. With this notation set, we define a \textit{hierarchical log-linear model} as follows (this is Definition 9.3.2 in \cite{Sull}).

\begin{defin}
Let $\Gamma\subseteq 2^{[m]}$ be a simplicial complex and let $r_1,\dots,r_m\in\mathbb{N}$. For each facet $F\in\Gamma$, let $\theta^{(F)}_{i_F}$ be a set of $\#\mathcal{R}_F$ \textit{positive} parameters. Then the \textbf{hierarchical log-linear model} associated with $\Gamma$ is the set of all probability distributions
\begin{center}
$\mathcal{M}_\Gamma:=\left\{p\in\Delta_{\mathcal{R}-1} : p_i=\displaystyle\frac{1}{Z(\theta)}\prod_{F\in\text{facet}(\Gamma)}\theta^{(F)}_{i_F}, \forall i\in\mathcal{R}\right\}$,
\end{center}
where $Z(\theta)$ is the normalizing constant
\begin{center}
$Z(\theta)=\displaystyle\sum_{i\in\mathcal{R}}\prod_{F\in\text{facet}(\Gamma)}\theta^{(F)}_{i_F}$.
\end{center}
\end{defin}

We will denote by $A_\Gamma$ the integer matrix that describes the hierarchical model,  as in the example at the end of \S1.1. Note that $A_\Gamma$ is a $\{0,1\}$-matrix, as monomials corresponding to facets must be squarefree. This paper mainly concerns the rank of $A_\Gamma$. A formula for this rank was first computed in \cite{HS}. We recall the result below.

\begin{prop}[Cor. 2.7 in \cite{HS} or Prop. 9.3.10 in \cite{Sull}]
Let $\Gamma$ be a simplicial complex on $[m]$, and $r_1,\dots,r_m\in\mathbb{N}$. The rank of the matrix $A_\Gamma$ associated to these parameters is
\begin{center}
$\displaystyle\sum_{F\in\Gamma}\prod_{f\in F}(r_f-1)$,
\end{center}
where the sum runs over all faces of $\Gamma$. The dimension of the associated hierarchical model $\mathcal{M}_\Gamma$ is one less than the rank of $A_\Gamma$.
\end{prop}

The original proof of this proposition was done by studying the kernel of $A_\Gamma$. In other words, the proof is based on the structure of the matrix as a linear map. In \S4, we show that the above formula can be recovered by focusing on the structure of $\Gamma$ with no appeal to the linear algebra of $A_\Gamma$. As the rank of $A_\Gamma$ determines the dimension of the model $\mathcal{M}_\Gamma$, it also determines the number of degrees of freedom of the model, as this is the codimension of the model in $\Delta_{\#\mathcal{R}-1}$.

Two particular models we are interested in are the \textit{main effect model}, where $\Gamma$ consists of the singleton sets and the empty set, and the \textit{saturated model}, where $\Gamma=2^{[m]}$. We think of these models as being on opposite extremes when it comes to conditional independence. The saturated model, in particular, is useful for defining the deviance of a model in hypothesis testing. See \cite{HEL} for more details.

\section{Exponential Hilbert series of simplicial complexes}

In this section, we define (both coarse and fine) exponential Hilbert series and present two theorems from \cite{JM} on how they are computed. The proofs of these two theorems are simple adaptations of the analogous theorems on classical Hilbert series, which can be found in \cite{MS}. First, let $\Gamma$ be a simplicial complex on $[m]$. We denote by $I_\Gamma$ the Stanley-Reisner ideal given by
\begin{center}
$I_\Gamma := (\textbf{x}^\textbf{a}\mid [a]\notin\Gamma)$.
\end{center}
In other words, the Stanley-Reisner ideal is the (squarefree) monomial ideal in $S=\mathbb{C}[x_1,\dots,x_m]$ generated by the non-faces of $\Gamma$. Then the \textit{Stanley-Reisner ring}, sometimes called the \textit{face ring}, of $\Gamma$ is the quotient of $S$ by $I_\Gamma$. We denote this ring as $SR_\Gamma$.

Since $I_\Gamma$ is generated by squarefree monomials, the Stanley-Reisner ring is graded over $\mathbb{N}^m$. For a given monomial $\textbf{x}^{\textbf{a}}$, let $M_{\textbf{a}}$ denote the $\textbf{a}^{\text{th}}$-graded component of $SR_\Gamma$. Then the \textit{(coarse) exponential Hilbert series} of $\Gamma$ is the formal power series
\begin{center}
$E(\Gamma;\textbf{x}):=\displaystyle\sum_{\textbf{a}\in\mathbb{N}^m}\dim(M_\textbf{a})\frac{\textbf{x}^\textbf{a}}{\textbf{a}!}$,
\end{center}
where $\textbf{a}!=a_1!a_2!\dots a_m!$. By setting each $x_i=t$, we obtain the fine exponential Hilbert series of $\Gamma$. We will see in the next section that the coarse series can be used to study hierarchical models on $\Gamma$ such that the individual random variables can have differing numbers of outcomes, while the fine series can be used to study the simpler case where each variable has the same number of outcomes. We present two theorems from \cite{JM} that provide formulas for each series for an arbitrary complex, $\Gamma$. Denote by $E(\Gamma;t)$ the fine exponential Hilbert series.

\begin{thm}
Let $\Gamma$ be a simplicial complex with Stanley-Reisner ideal $I_\Gamma$. Then
\begin{center}
$E(\Gamma;\textbf{x}) = \displaystyle\sum_{F\in\Gamma}\prod_{f\in F}(e^{x_f}-1)$.
\end{center}
\end{thm}

A similar theorem for the fine exponential Hilbert series follows quickly from the above.

\begin{thm}
Let $\Gamma$ be a simplicial complex with Stanley-Reisner ideal $I_\Gamma$, then
\begin{center}
$E(\Gamma;t)=\displaystyle\sum_{F\in\Gamma}(e^t-1)^{\dim(f)+1}$.
\end{center}
\end{thm}

In particular, $E(\Gamma;t)$ is polynomial in $e^t$. Let $d=\dim(\Gamma)+1$. Then, if we set
\begin{center}
$E(\Gamma;t)=E_0+E_1e^t+\dots+E_de^d$,
\end{center}
we have the following characterization of $E_0, \dots, E_d$, in terms of the face vector.
\begin{equation}
E_k=\displaystyle\sum_{i=k}^d(-1)^{i-k}f_{i-1}{i \choose k}.
\end{equation}
The vector $(E_0,\dots,E_d)$ is called the $e$-\textit{vector} of $\Gamma$, as it relates to the $h$-vector and $f$-vectors from classical discrete geometry. See \cite{JM} for details.

In the next section, we will use these theorems along with the result from \cite{Sull} in \S2 to show how the rank of $A_\Gamma$ can be computed using exponential Hilbert series.

\section{Main theorems}

\begin{thm1}
Let $X=(X_1,\dots,X_m)$ be a random vector with $X_i$ having $r_i$ possible outcomes. Let $\Gamma$ be an abstract simplicial complex and let $\mathcal{M}_{A_{\Gamma}}$ be the hierarchical log-linear for $X$ on $\Gamma$. Then the rank of the matrix $A_\Gamma$ of sufficient statistics is
\begin{center}
$E(\Gamma;\log(r_1),\dots,\log(r_m))$.
\end{center}
\end{thm1}
\begin{proof}
The proof of this result is really just a combination of two previous theorems. Recall that, from \cite{JM}, we have
\begin{equation}
E(\Gamma;\textbf{x})=\displaystyle\sum_{F\subseteq\Gamma}\prod_{f\in F}(e^{x_f}-1).
\end{equation}
We also have, from \cite{Sull},
\begin{equation}
\text{rank}(A_\Gamma)=\displaystyle\sum_{F\in\Gamma}\prod_{f\in F}(r_f-1).
\end{equation}
Then substituting $x_f\mapsto\log(r_i)$ into (4.1) yields (4.2).
\end{proof}

\begin{thm2}
Assume that $r_1=r_2=\dots=r_m=r$. Then the rank of $A_\Gamma$ is
\begin{center}
$E(\Gamma;\log(r))=E_0+E_1r+\dots+E_dr^d$,
\end{center}
where $d=\dim(\Gamma)$.
\end{thm2}

\begin{proof}
Let $r_1=r_2=\dots=r_m=r$. Then,
\begin{center}
$\text{rank}(A_\Gamma)=E(\Gamma;\log(r),\dots,\log(r))$.
\end{center}
Plugging $\log(r)$ in for each variable in Theorem 1 yields
\begin{center}
$E(\Gamma;\log(r),\dots,\log(r))=\displaystyle\sum_{F\subset\Gamma}\prod_{i\in F}(e^{\log(r)}-1)=\sum_{F\subset\Gamma}\prod_{i\in F}(r-1)$.
\end{center}
Expanding the right-hand side over the product yields
\begin{center}
$\displaystyle\sum_{F\subset\Gamma}(r-1)^{\dim(F)+1}=E(\Gamma;\log(r))$.
\end{center}
\end{proof}

\begin{cor}
Assume that $r_1=r_2=\dots=r_m=r$, and let $d=\dim(\Gamma)+1$. Then the rank of $A_\Gamma$ is
\begin{center}
$\text{rank}(A_\Gamma)=\displaystyle\sum_{k=0}^d\sum_{i=k}^d(-1)^{i-k}f_{i-1}{i \choose k}r^k$.
\end{center}
\end{cor}
\begin{proof}
This follows immediately from Theorem 2 and (3.1).
\end{proof}

As a quick application of Corollary 1, we consider cyclic models. In a \textit{cyclic} model on $[m]$, $\Gamma=[12][23]\dots[(m-1)m][m1]$. Thus, the face vector of $\Gamma$ is $(1,m,m)$, as $\Gamma$ is a graph with $m$ vertices and $m$ edges. The $e$-vector of $\Gamma$ is easily computed via (3.1) to be $(1,-m,m)$. The following then follows immediately from Theorem 2/Corollary 1.

\begin{prop}
Let $\Gamma=[12][23]\dots[(m-1)m][m1]$, and assume $r_1=r_2=\dots=r_m=r$. Then
\begin{center}
$\text{rank}(A_\Gamma)=1-mr+mr^2$.
\end{center}
\end{prop}

Note that Corollary 1 provides a quick way of computing the rank of $A_\Gamma$ in many other examples, provided each random variable has the same outcomes. For example, let $\Gamma=[12][14][23]=\{\emptyset, \{1\},\{2\},\{3\},\{4\},\{1,2\},\{1,4\},\{2,3\}\}$. Then the face vector of $\Gamma$ is $f=(1,4,3)$, if $r_1=r_2=r_3=r_4=2$, then
\begin{equation}
\text{rank}(A_\Gamma)=\displaystyle\sum_{k=0}^2\sum_{i=k}^2(-1)^{i-k}f_{i-1}{i\choose k}r^k.
\end{equation}
In the notation of \S3, $E_0=0$, $E_1=-2$, and $E_2=3$. Then (4.3) becomes
\begin{center}
$\text{rank}(A_\Gamma)=0-2r+3r^2=8$.
\end{center}
For another quick example, let $\Gamma$ be the boundary of the tetrahedron, which has face vector $f=(1,4,6,4)$. In this case, we have $E_0=-1$, $E_1=4$, $E_2=-6$, and $E_3=4$. This gives
\begin{center}
$\text{rank}(A_\Gamma)=-1+4r-6r^2+4r^3$.
\end{center}
If all of the variables are binomial, then we have $\text{rank}(A_\Gamma)=15$. Note that the boundary of the tetrahedron corresponds to the model in which only the full interaction factor between all of the variables is not considered.

In the second example above, notice that the values of $E_i$, when put in order, are the values of the face vector, but alternating. This is not true of the first example. In \cite{JM}, this relationship is explored, and it is shown that this is the case if and only if the simplicial complex satisfies the Dehn-Sommerville equations (for an overview of the Dehn-Sommerville equations, see \S3.16 of \cite{RS}). This motivates the following definition.

\begin{defin}
A \textbf{DS-hierarchical log-linear model} is a hierarchical model on a simplicial complex satisfying the Dehn-Sommerville equations.
\end{defin}

In the case of DS-hierarchical log-linear models, the computation of the rank is even simpler.

\begin{cor}
If $\mathcal{M}_\Gamma$ is a DS-hierarchical log-linear model, and $r_1=r_2=\dots=r_m=r$, then we have
\begin{center}
$\text{rank}(A_\Gamma)=\displaystyle\sum_{i=0}^d(-1)^{d-i}f_{i-1}r^i$.
\end{center}
\end{cor}
\begin{proof}
By Theorem 3 in \cite{JM}, the Dehn-Sommerville equations are equivalent to the statement that $E_i=(-1)^{d-i}f_{i-1}$ for all $i=0,\dots,d$.
\end{proof}

If $\mathcal{M}_\Gamma$ is a DS-hierarchical log-linear model, and all random variables in $X$ have the same number of outcomes, then $\text{rank}(A_\Gamma)$ (and thus the dimension of $\mathcal{M}_\Gamma$) is particularly simple to compute. In this case, it is the structure of the simplicial complex that is key to determining the rank.

Returning to the cyclic model, note that $f=(1,m,m)$, and $E=(1,-m,m)$. Since the dimension $\Gamma$ is one, and thus $d=2$, we have $E_0=(-1)^{2-0}1$, $E_1=(-1)^{2-1}m$, and $E_2=(-1)^{2-2}m$. Thus, the cyclic model is another example of a DS-hierarchical model.

\section{Further examples: The main effect and saturated models}

We compute the rank of $A_\Gamma$ using the exponential Hilbert series for the main effect model and the saturated model. While formulas for the dimension and degrees of freedom for these models are not new, these computations showcase how the new formulas in \S4 can be used. Note that neither of these models is a DS-hierarchical model.

\begin{example}[The Main Effect Model] Let $\Gamma$ be the main effect model such that $\Gamma=[1][2]\dots[m]$. Then we have the following formulas:
\begin{equation}
E(\Gamma;\textbf{x})= 1 + (e^{x_1}-1) + \dots + (e^{x_m}-1)
\end{equation}
\begin{equation}
E(\Gamma;t)= 1+m(e^t-1)
\end{equation}
If $X=(X_1,\dots,X_m)$, and all $X_i$ have the same number of outcomes, $r$, then by Theorem 2 and (5.2), we have
\begin{center}
$\text{rank}(A_\Gamma)=1-m+mr$.
\end{center}
If we allow the ranks to differ, then we use Theorem 1 and (5.1) to get
\begin{center}
$\text{rank}(A_\Gamma)=1-m+\displaystyle\sum_{i=1}^m r_i$.
\end{center}
In the case where all $X_i$ have $r$ outcomes, the dimension of the model is $m(r-1)$. The probability simplex $\Delta_{r^m}$ has dimension $r^m-1$. Thus, there are $r^m-1-m(r-1)$ degrees of freedom.
\end{example}

\begin{example}[The Saturated Model] Let $\Gamma$ be the saturated model corresponding to $2^{[m]}$. Consider the case where each variable has $r$ outcomes. Then, as $\Gamma=2^{[m]}$, the number of $i$-faces of $\Gamma$ is simply $\displaystyle{m \choose i}$. Then,
\begin{center}
$E_k=\displaystyle\sum_{i=k}^m(-1)^{i-k}{m\choose i}{i\choose k}=\sum_{i=k}^m(-1)^{i-k}{m\choose k}{m-k\choose i-k}$.
\end{center}
From this formula, we immediately get that $E_m=1$. For a fixed $k\neq m$, we have
\begin{center}
$E_k=\displaystyle{m\choose k}\sum_{i=k}^m(-1)^{i-k}{m-k\choose i-k}$.
\end{center}
After making the substitution $l=m-k$ in the above sum, we get
\begin{center}
$E_k=\displaystyle{m\choose k}\sum_{i=m-l}^m(-1)^{i-(m-l)}{l-k\choose i-(m-l)}$.
\end{center}
The sum above is equal to $\displaystyle\sum_{j=0}^l(-1)^j{l\choose j}=0$. Thus, if $k\neq m$, we have $E_k=0$. Then the rank of $A_\Gamma$ is
\begin{center}
$\text{rank}(A_\Gamma)=r^m$.
\end{center}
Thus, the dimension of the saturated model is $r^m-1$, which is the same as the probability simplex $\Delta_{r^m}$. The exponential Hilbert series then recovers the fact that the saturated model has zero degrees of freedom.
\end{example}

As a final example, let $\Gamma$ be the boundary of the saturated model. Therefore $\Gamma$ is the model that includes all interaction factors except for the interaction term between all variables. $\Gamma$ has the same number of $i$-faces as the saturated model, except $\Gamma$ has no $m-1$ face. Since $\Gamma$ is the boundary of a simplicial polytope, it is a DS-hierarchical model (see \cite{RS2}). Therefore, we have 
\begin{center}
$E_k=(-1)^{k+1}f_{k-1}$, 
\end{center}
for $k=0,\dots,m-1$, which yields
\begin{center}
$\text{rank}(A_\Gamma)=\displaystyle\sum_{i=0}^{m-1}(-1)^{i+1}{m\choose i}r^i$,
\end{center}
assuming each random vector has $r$ outcomes.

\section{Open questions}

We end with some open questions regarding Hilbert series and log-linear hierarchical models. We have seen before that computing the rank of $A_\Gamma$ for a DS-hierarchical model is relatively simple, assuming that each random vector has the same number of outcomes. However, the definition of a DS-hierarchical model is combinatorial and not statistical.

\begin{prob}
Find a statistical explanation for the DS-hierarchical property.
\end{prob}

Throughout the paper, we also only considered the rank of $A_\Gamma$ for a DS-hierarchical model when we considered random variables with the same number of outcomes. This leads us to the following question.

\begin{prob}
Is it possible to find a simple description of $\text{rank}(A_\Gamma)$ for a DS-hierarchical model if we allow the number of outcomes to vary?
\end{prob}

Finally, many hierarchical models are graphical models, where the complex $\Gamma$ is the clique complex of a certain graph (see Chapter 13 of \cite{Sull}) for details. The classical Hilbert series of the clique complex was computed in \cite{FF} in terms of the subgraph polynomial.

\begin{prob}
Find a way to relate the Hilbert series of the clique complex and/or the subgraph polynomial to the statistics of hierarchical models.
\end{prob}

There has been very little research done at this point into how Hilbert series (both classical and exponential) of simplicial complexes relate to hierarchical models. Given the interplay between properties of the simplicial complex and properties of the model through the exponential Hilbert series, this line of questioning seems promising.


\begin{bibdiv}
\begin{biblist}



\bib{FF}{article}{
   author={Ferrarello, Daniela},
   author={Fr\"{o}berg, Ralf},
   title={The Hilbert series of the clique complex},
   journal={Graphs Combin.},
   volume={21},
   date={2005},
   number={4},
   pages={401--405},
}


\bib{HEL}{book}{
   author={H\o jsgaard, S\o ren},
   author={Edwards, David},
   author={Lauritzen, Steffen},
   title={Graphical Models with R},
   series={Use R!},
   publisher={Springer Science+Business Media, New York},
   date={2012},
}

\bib{HS}{article}{
   author={Ho\c{s}ten, Serkan},
   author={Sullivant, Seth},
   title={Gr\"{o}bner bases and polyhedral geometry of reducible and cyclic
   models},
   journal={J. Combin. Theory Ser. A},
   volume={100},
   date={2002},
   number={2},
   pages={277--301},
}

\bib{JM}{arXiv}{
   author={Johnson, Wayne A.},
   author={Mogilski, Wiktor J.},
   title={The $e$-vector of a simplicial complex},
   date={2018},
   eprint={1806.05239},
   archiveprefix={arXiv},
   primaryclass={math.CO},
}

\bib{MS}{book}{
   author={Miller, Ezra},
   author={Sturmfels, Bernd},
   title={Combinatorial commutative algebra},
   series={Graduate Texts in Mathematics},
   volume={227},
   publisher={Springer-Verlag, New York},
   date={2005},
}

\bib{NS}{article}{
   author={Novik, Isabella},
   author={Swartz, Ed},
   title={Applications of Klee's Dehn-Sommerville Relations},
   journal={Discrete and Computational Geometry},
   volume={42},
   date={2009},
   number={2},
   pages={261--276},
}

\bib{RS}{book}{
   author={Stanley, Richard P.},
   title={Enumerative Combinatorics, Volume 1},
   edition={Second Edition},
   series={Cambridge Studies in Advanced Mathematics},
   volume={49},
   publisher={Cambridge University Press, New York},
   date={2012},
}

\bib{RS2}{book}{
   author={Stanley, Richard P.},
   title={Combinatorics and Commutative Algebra},
   edition={Second Edition},
   series={Progress in Mathematics},
   volume={41},
   publisher={Birkhauser, Boston},
   date={1996},
}

\bib{Sull}{book}{
   author={Sullivant, Seth},
   title={Algebraic Statistics},
   series={Graduate Studies in Mathematics},
   volume={194},
   publisher={American Mathematical Society, Rhode Island},
   date={2018},
}

\end{biblist}
\end{bibdiv}
\end{document}